\documentclass[11pt]{article}

\usepackage{amssymb}
\usepackage{amscd}
\usepackage{amsmath}
\usepackage{amsfonts}
\usepackage{theorem}
\usepackage{mathrsfs}
\usepackage{tikz}
\usepackage{graphicx}
\usepackage{float}
\usepackage{color}
\definecolor{blue}{rgb}{0,0,1}

{\theorembodyfont{\slshape}
\newtheorem{theorem}{Theorem}
\numberwithin{theorem}{section}

\newtheorem{lemma}[theorem]{Lemma}
\newtheorem{corollary}[theorem]{Corollary}

}
{\theorembodyfont{\rmfamily}

\newtheorem{remark}[theorem]{Remark}

}

\def\I{{\mathcal{I}}}
\def\J{{\mathcal{J}}}
\def\O{{\mathcal{O}}}
\def\V{{\mathcal{V}}}

\def\N{{\mathbb{N}}}
\def\C{{\mathbb{C}}}

\renewcommand{\P}{\mathbb{P}}

\renewcommand{\tilde}{\widetilde}

\newcommand{\algoritmo}{\begin{minipage}{0.87\hsize}\linea}
\newcommand{\falgoritmo}{\linea\end{minipage}\bigskip}
\newcommand{\linea}{\vspace*{-5pt}\hrule\vspace*{5pt}}

\medskip

\begin{document}

\bibliographystyle{plain}

\makeatletter


\def\JACM{Journal of the ACM}
\def\CACM{Communications of the ACM}
\def\ICALP{International Colloquium on Automata, Languages
            and Programming}
\def\STOC{annual ACM Symp. on the Theory
          of Computing}
\def\FOCS{annual IEEE Symp. on Foundations of Computer Science}
\def\SIAM{SIAM Journal on Computing}
\def\SIOPT{SIAM Journal on Optimization}
\def\MOR{Math. Oper. Res.}
\def\BSMF{Bulletin de la Soci\'et\'e Ma\-th\'e\-ma\-tique de France}
\def\CRAS{C. R. Acad. Sci. Paris}
\def\IPL{Information Processing Letters}
\def\TCS{Theoretical Computer Science}
\def\BAMS{Bulletin of the Amer. Math. Soc.}
\def\TAMS{Transactions of the Amer. Math. Soc.}
\def\PAMS{Proceedings of the Amer. Math. Soc.}
\def\JAMS{Journal of the Amer. Math. Soc.}
\def\LNM{Lect. Notes in Math.}
\def\LNCS{Lect. Notes in Comp. Sci.}
\def\JSL{Journal for Symbolic Logic}
\def\JSC{Journal of Symbolic Computation}
\def\JCSS{J. Comput. System Sci.}
\def\JoC{J. of Complexity}
\def\MP{Math. Program.}
\sloppy

\begin{title}
{{\bf The polynomial eigenvalue problem is well conditioned for random inputs.}} 
\end{title}
\author{
  Diego Armentano\thanks{Partially supported by Agencia Nacional de 
  Investigaci\'on e Innovaci\'on (ANII), Uruguay, and by CSIC group 618}\\
  Universidad de La Rep\'ublica\\
  URUGUAY\\
  {\tt diego@cmat.edu.uy}
  \and
  Carlos Beltr\'an\thanks{ partially suported by MTM2014-57590-P from 
  Spanish Ministry of Science MICINN}\\
  Universidad de Cantabria\\ 
  SPAIN\\
  {\tt beltranc@unican.es}
  }

\date{\today}

\makeatletter
\maketitle
\makeatother

\thispagestyle{empty}

\begin{quote}
{\small 
{\bf Abstract.} 
We compute the exact value of the squared condition number for the polynomial eigenvalue problem, when the input matrices have entries coming from the standard complex Gaussian distribution, showing that in general this problem is quite well conditioned.
}
\end{quote}

\section{Introduction}
Recall the (homogeneous) generalized eigenvalue problem, GEPV: given two $n\times n$ matrices $A,B$, find $(\alpha,\beta)\in\P(\C^2)$ such that
\[
 \det(\beta A-\alpha B)=0.
\]
The point $(\alpha,\beta)$ is called a (generalized) eigenvalue of $(A,B)$ and the  corresponding vectors $x,y\in\C^n$ satisfying
\[
 (\beta A-\alpha B)x=0,\quad y^*(\beta A-\alpha B)=0
\]
are called the right and left eigenvectors of $(A,B)$. 

The polynomial eigenvalue problem (PEVP), a well known generalization of the GEVP, is among the most important problems in Numerical Linear Algebra, see the excellent survey \cite{MehrmannVoss} for context. The question now is: given $A_0,\ldots,A_d\in\C^{n\times n}$, find $(\alpha,\beta)\in\P(\C^2)$ such that
\[
\mathcal{F}(A,(\alpha,\beta))=\det(P(A,\alpha,\beta))=0,
\]
where $A=(A_0,\ldots,A_d)$ and
\begin{equation}\label{eq:PEVP}
  P(A,\alpha,\beta)=\beta^dA_0+\alpha\beta^{d-1}A_1+\cdots+\alpha^{d-1}\beta A_{d-1}+\alpha^dA_d=\sum_{k=0}^d\alpha^k\beta^{d-k}A_k.
\end{equation}

{A solution $(\alpha,\beta)\in\P(\C^2)$  of
$ P(A,\alpha,\beta)=0$ is called a generalized eigenvalue, or simply an eigenvalue, of $A$. 
Let us denote by $\mathrm{Eig}(A)$ the set of eigenvalues of $A$.}

For generic input $A$, there exist $nd$ such eigenvalues.  Right and left eigenvectors of $A$ are defined similarly to the GEVP case. If we set $d=1$ we get (up to a sign convention) the GEVP problem.

A good deal of effort has been dedicated in the last years to find efficient algorithms for particular instances of the PEVP such as the quadratic eigenvalue problem, see \cite{BaiSu,Meerbergen,Tisseur2013} (the inverse version \cite{Zaballa} is also of interest). The nowadays method of choice for solving the PEVP is to linearize it obtaining a GEVP (in dimension $nd$) that can be  solved using standard eigenvalue solvers. Different ways to linearize the problem can be analyzed in search for optimal ones in terms of stability or numerical convenience, see \cite{Higham1}.

A fundamental question regarding the numerical solution of the PEVP is the stability, governed by the so--called condition number \cite{DeTi,Higham2,Higham3}. The condition number for the GEVP has a nice expression computed in \cite[Section 6]{DeTi}: for a matrix pair $(A,B)\in\C^{n\times n}\times \C^{n\times n}$ and an eigenvalue $(\alpha,\beta)\in\P(\C^2)$, we have:
\[
 \mu((A,B),(\alpha,\beta))=\frac{\|(\alpha,\beta)\|\,\|x\|\,\|y\|}{|\bar\alpha y^*Ax+\bar\beta y^*Bx|}\,\|(A,B)\|_F,
\]
where $x$ and $y$ are the corresponding right and left eigenvectors, and $\|\cdot\|_F$ denotes Frobenius norm (the factor $\|(A,B)\|_F$ is missing in \cite{DeTi} since Dedieu and Tisseur compute the absolute condition number instead of the relative condition number). This formula is indeed a consequence of the definition of the condition number as the maximum of the change in the eigenvalue when the input is locally perturbed (see Section \ref{sec:general} for a more formal definition).

The same definition as the maximum possible change in the eigenvalue when the input is perturbed is valid for the more general PEVP. An explicit formula for the condition number for the PEVP was derived in \cite[Th. 4.2]{DeTi}:
\[
 \mu(A,(\alpha,\beta))=\left(\sum_{k=0}^d|\alpha|^{2k}|\beta|^{2d-2k}\right)^{1/2}\frac{\|x\|\,\|y\|}{|y^*v|}\,\|A\|_F,
\]
where again $A=(A_0,\ldots,A_d)$,
{$(\alpha,\beta)\in\mathrm{Eig}(A)$, and} $x$
and $y$ are the {corresponding} right and left eigenvectors and
\begin{equation}\label{eq:v}
 v=\bar\beta \frac{\partial}{\partial \alpha} P( A,\alpha,\beta)x-\bar\alpha \frac{\partial}{\partial \beta}P( A,\alpha,\beta)x.
\end{equation}
Note that the condition number depends both on the input $A$, and on the particular eigenvalue $(\alpha,\beta)$. It is customary to consider
\[
  \mu_{\mathrm max}(A)=\max_{{(\alpha,\beta)\in \mathrm{Eig}(A)}}\mu(A,(\alpha,\beta)).
\]

Computing the condition number for any numerical problem (including the PEVP) is a time-consuming task that suffers from intrinsic stability problems as pointed out in \cite{Demmel}. It is hence usual to estimate average values of the condition number for a given family of inputs, in such a way that we can rely on probabilistic arguments instead of computing conditions of particular inputs. The linear algebra case is probably the most studied one, see \cite{Edelman,ChenDongarra2005} for estimates on Turing's condition number of linear algebra with different normalizations. Theoretical results describing the average condition number of the standard eigenvalue--eigenvector problem have been recently obtained \cite{ArmCuc,ABBCS}, but we are not aware of any previous similar result for the PEVP case. In this paper we fill this gap:
\begin{theorem}\label{th:main}
 Let $A=(A_0,\ldots,A_d)\in\C^{(d+1)n^2}$ be chosen at random with
  {independent} entries  following {the standard complex Gaussian
  distribution $\mathcal{N}_\C(0,1)$}. Then, the expected value of the squared condition number for the PEVP satisfies:
\[
  {\rm
  E}_{A}\left(\frac{1}{dn}\sum_{{(\alpha,\beta)\in \mathrm{Eig}(A)}}\mu(A,(\alpha,\beta))^2\right)=\frac{(d+1)n^2-1}{d}.
  \]
\end{theorem}
 A trivial consequence of Theorem \ref{th:main} is:
\begin{corollary}\label{cor:maxmu}
Let $A=(A_0,\ldots,A_d)\in\C^{(d+1)n^2}$ be chosen at random with
  {independent} entries {with distribution
  $\mathcal{N}_\C(0,1)$}. Then,
\[
 {\rm E}_{A}\left(\mu_{\mathrm max}(A)^2\right)\leq((d+1)n^2-1)n,
\]
and in particular
\[
 {\rm E}_{A}\left(\log\left(\mu_{\mathrm max}A)\right)\right)\leq \frac32\log n+\frac12\log(d+1).
\]
\end{corollary}
Since the number of digits needed to describe accurately the output of a numerical problem is controlled by the logarithm of the condition number, we can now see from Corollary \ref{cor:maxmu} that there is no intrinsic obstruction for computing solutions to the PEVP, even for very large values of $n$ and $d$.

The proof of Theorem \ref{th:main} (see Section \ref{PEVP}) will be a consequence of a much more general result, Theorem \ref{th:general}. We also prove similar results for other numerical problems, see sections \ref{Dense}, \ref{Lacunary}, \ref{GEVP}, \ref{sparsePEVP}, \ref{ellaboratedPEVP} for other examples of application.
\begin{remark}
 Since the condition number satisfies $\mu(tA,(\alpha,\beta))=\mu(A,(\alpha,\beta))$ for any nonzero $t\in\C$, the results above also apply in the case that the coefficients of the input matrices are $\mathcal{N}_\C(0,\sigma)$ for $\sigma\in(0,\infty)$, as far as all the coefficients follow the same distribution.
\end{remark}

\begin{remark}
  The standard approach to solve a polynomial equation $p(x)=0$ {in one
  complex variable,} is to construct the companion matrix of the given polynomial and compute its eigenvalues (this is the way used by {\sc Matlab}'s command {\tt root}). It is usually convenient to balance the companion matrix before applying QR or some other standard eigensolver. Specific eigensolvers that take advance of the special structure of companion matrices have also been developped, see \cite{Bini1,Bini2,Barel,Aurentz} and references therein. The stability of the process can be improved by using a companion matrix different from the standard one, see  \cite{Dopico}.
 
 Currently existing eigenvalue solvers exhibit remarkable stability and accuracy properties, and indeed the above process shows excellent practical performance. In contrast, computing eigenvalues of a matrix writing down its characteristic polynomial and solving by Newton's method or other basic procedure shows quite a poor performance in practice. This may have contributed to support the idea, extended in some part of the community, that polynomial root finding is poor--conditioned, while eigenvalue computation is well--conditioned. Indeed, in \cite[p. 92]{Trefethen} root finding is presented as a classic example of ill-conditioned problem, and so is the computation of eigenvalues in the non--symmetric case. 
 
 Theorem \ref{th:main} shows that in a homogeneous context, at least if the coefficients of the input instances are derived from $\mathcal{N}_\C(0,1)$, we actually have the following.
 \begin{quote}
{\em Both  polynomial root finding and polynomial eigenvalue problems are pretty well conditioned on the average, and indeed polynomial root finding has average squared condition number exactly equal to $1$.}
   \end{quote}
\end{remark}
\begin{remark}
 The condition number in a numerical problem measures the change in the solution when the input is locally perturbed, giving a worst case bound: the worst (local) perturbation in the input defines the condition of a problem.
 
 In some situations, perturbations of the input can be expected to be ``random'' in the sense that there is no prefered direction, see \cite{Armentano:10}. Of course, in this case the observed condition number of the problem may seem better than the theoretical bound, since most perturbations will be quasiorthogonal to the worst direction from the concentration of measure phenomenon. The condition number for random direction in the perturbation was called in \cite{Armentano:10} the {\em stochastic condition number} $\mu_{\mathrm st}$. Its formula for the PEVP is given (using \cite[Th. 4.2]{DeTi}) by:
 \[
  \mu_{\mathrm st}(A,(\alpha,\beta))^2=\mathrm{E}_{\dot A}\left(\frac{|y^*P(\dot A,\alpha,\beta)x|^2}{|y^*v|^2}\,\|A\|_F^2\right),
 \]
 where $\dot A=(\dot A_0,\ldots, \dot A_d)$ lies in the unit sphere in $\C^{n^2(d+1)}$ and $v$ is given by \eqref{eq:v}. From \cite[Theorem 1]{Armentano:10} we have $\mu^2=(d+1)n^2\mu_{\mathrm st}^2$, which gives the equality 
 \[
   {\rm
   E}_{A}\left(\frac{1}{dn}\sum_{{(\alpha,\beta)\in\mathrm{Eig}(A)}}\mu_{\mathrm st}(A,(\alpha,\beta))^2\right)=\frac{(d+1)n^2-1}{d(d+1)n^2}\approx\frac1d.
 \]
 Thus, for a random eigenvalue of a random input and a random direction in the perturbation, one should expect the relative change in the solution to be around $\sqrt{1/d}$ times the size of the perturbation (relative to $\|A\|_F=\|(A_0,\ldots,A_d)\|_F$).

\end{remark}

\section{Geometric framework: a general result}\label{sec:general}
We will follow the path of \cite{ShSm93b,DeTi,Tisseur2000} but considering a general setting of input and output spaces which can be applied to many different situations 
{(see for example \cite{BlCuShSm98}, or \cite{Armentano:10})}.

In many numerical problems there is a space of inputs $\I$ that we can
identify with $\C^m$, and a space of outputs $\O$ that we can identify
with a one--dimensional projective irreducible algebraic subvariety of $\P(\C^b)$,
for some $b\in\N$. We denote by $d_\O$ the degree of $\O$. We denote elements in $\I$ by $p$ and elements in $\O$ by $z$, and we consider a polynomial $\mathcal{F}:\C^m\times\C^b\to\C$

\[
 \mathcal{F}(p,z)
\]
bihomogeneous in its two variables with degrees:
\[
 \deg_p \mathcal{F}=r,\quad \deg_z\mathcal{F}=s.
\]
We look at the problem:
\begin{center}
 Given $p\in\I$, find $z\in\O$ such that $\mathcal{F}(p,z)=0$.
\end{center}
Then, we consider the algebraic variety
\[
 \V=\{(p,z)\in\I\times\O:\mathcal{F}(p,z)=0\}\subseteq\C^m\times\P(\C^b),
\]
as well as the two natural projections $\pi_1:\V\to\I$, $\pi_2:\V\to\O$. 

{The space $\C^m$ is equipped with the canonical Hermitian inner
product, and $\mathbb{P}(\C^b)$ is equipped with the induced Fubini--Study metric and Hermitian
structure (see \cite{BlCuShSm98} for example).
}

A random choice of $p\in\I$ is to be understood as a random variable for the Gaussian density
\[
 \frac{1}{\pi^m} e^{-\|p\|^2}.
\]
{
That is, under the equivalence $\I\equiv\C^m$, the elements of $\I$
follow the standard complex Gaussian distribution.
}

We will use the following notation for norms:
\begin{itemize}
  \item By $\|\cdot\|$ we mean the {Euclidean} vector norm.
    {In particular, if $A$ is a matrix, we have $\|A\|=\|A\|_F$  (Frobenius norm).}
 \item By $\|A\|_2$ (with $A$ a matrix or a linear operator) we denote the operator $2$--norm of $A$.
 \item {By $\|\cdot\|_z$ we mean the norm in the tangent space
   $T_z\P(\C^b)$.}
\end{itemize}
Here are some basic examples of the general setting above, including the PEVP case:
\begin{enumerate}
 \item[\bf A] $\I\equiv\C^{N+1}$ is the set of polynomials of degree at most $N$, homogeneous in two variables, and we search for a zero $z\in\P(\C^2)$. Here,
 \[
  \mathcal{F}(p,z)=p(z).
  \]
 \item[\bf B] $\I\equiv\C^{n\times n}\times \C^{n\times n}\equiv
   \C^{2n^2}$ is the set of pairs of matrices $(A,B)$, and we search for
    a generalized eigenvalue $(\alpha,\beta)\in\P(\C^2)$, i.e. {a
    solution of}
    $$\mathcal{F}((A,B),(\alpha,\beta))=\det(\beta A-\alpha B)=0.$$
    In general, there exist $n$ generalized eigenvalues.
 \item[\bf C] $\I\equiv\C^{n\times n}\times\cdots\times \C^{n\times n}\equiv\C^{(d+1)n^2}$ (there are $d+1$ copies of $\C^{n\times n}$), and we want to solve the PEVP, that is given $A=(A_0,\ldots,A_d)$ we want to find $(\alpha,\beta)\in\P(\C^2)$ such that \[
 \mathcal{F}(A,(\alpha,\beta))=\det(P(A,\alpha,\beta))=0,
 \]
{where $P$ is given in (\ref{eq:PEVP}).}
 In general, there exist $dn$ eigenvalues.
 Note that {\bf{A}} and {\bf{B}} can be seen as a particular cases of {\bf{C}}.
 \item[\bf D] One can also consider sparse or lacunary versions of the problems above, namely the same problems where some of the entries of the matrices, or some of the coefficients of the polynomials, are set to $0$.
\end{enumerate}
The relative condition number in these and other numerical problems is
defined as follows {(see \cite{BlCuShSm98}, \cite{Armentano:10}, or
\cite[Sec. 14.1]{Condition} for a general setting}, or \cite{DeTi} for an specification to the PEVP). 
Let $z_0$ be a solution for an input $p_0$. Assume that $z_0$ is a smooth point of $\O$ and let $\dot z_0$ be any nonzero vector in $T_{z_0}\O$. When the directional derivative
\[
D\mathcal{F}(p_0,z_0)(0,\dot z_0)
\]
is not equal to $0$, from the implicit function
theorem, there is a locally defined {\em solution map}, {denoted by} ${\mathrm Sol}$, which sends an input $p$ (close to $p_0$) to its output $z$ (close to $z_0$). 
{This map is given by the composition $\pi_2\circ\pi_1^{-1}$, locally defined in
a neighborhood of $p_0$.}
The condition number at $(p_0,z_0)$ is then defined by the
{operator} norm of the derivative of the solution map, normalized
by the {norm of the input}
{
\[
 \mu(\mathcal{F},p_0,z_0)=\|p_0\|\|D{\mathrm Sol}(p_0,z_0)\|_2=
\sup_{\dot p\in T_{p_0}\C^m}\frac{\|\dot z\|/\|{z_0}\|}{\|\dot p\|/\|p_0\|},
\]
where $\dot z=D\mathrm Sol (p_0,z_0)\dot p$.
(We will write $\mu(p_0,z_0)$  when the mapping $\mathcal{F}$ is
clear from the context.) }
%

In the case that ${\mathrm Sol}$ is just Lipschitz, the condition number
uses the local Lipschitz constant instead of the $2$--norm of the
derivative, but we will not deal with this case here. The condition
number is set to $\infty$ if $D\mathcal{F}(p_0,z_0)(0,\dot z_0)=0$ for some (i. e. for all) nonzero vector $\dot z_0\in T_{z_0}\O$.

This geometric definition of the condition number is inspired in the intuitive definition: it is a local bound for the change in the output under perturbations on the input, both measured in relative error terms. The classical condition number of linear algebra $\kappa(A)=\|A\|_2\,\|A^{-1}\|_2$ does not exactly follow this definition since in our general frame we measure the relative error with respect to the norm of the input as a vector. Indeed, our definition gives the so called Demmel's condition number:
\[
 \mu(A,x)=\|A\|_F\|A^{-1}\|_2,
\]
for the problem of solving $Ax=b$ with $b$ fixed. 

Our main result is a general formula for the expected value of the condition number. Its main feature is that it is valid for the general input-output setting described above. 
{For $p\in\I$, and $z\in\mathcal{O}$, let $n(p)=\sharp
\pi_1^{-1}(p)$, and $\V_z=\pi_2^{-1}(z)$, i.e.
\[
  n(p)=\sharp\{z\in\O:\mathcal{F}(p,z)=0\},\quad
  \V_z=\{p\in\I:\mathcal F(p,z)=0\},
\]
where we set $n(p)=+\infty$ in case $\pi_1^{-1}(p)$ not finite.
}

\begin{theorem}\label{th:general}
 With the notations above, assume that :
 \begin{itemize}
  \item $n(p)$ is finite for all $p\in\I$ out of some $m-2$ dimensional subvariety, and there exists $p_0\in\I$ such that
\begin{equation}\label{eq:condnew}
  n(p_0)=sd_\O.
\end{equation}
\item $\V_z$ is an $m-1$ dimensional variety for all $z\in\O$, and there exists $z_0\in\O$ such that the degree of $\V_{z_0}$ is equal to $r$.
 \end{itemize}

  Then, for all $p\in\I$ out of some zero--measure set, the equation $\mathcal{F}(p,z)=0$ has exactly $sd_\O$ solutions in $\O$. Moreover, the expected squared condition number satisfies:
  \begin{equation}\label{eq:main}
   {\rm E}_{p\in\I}\left(\frac{1}{sd_\O}\sum_{z:\mathcal{F}(p,z)=0}\mu(p,z)^2\right)=\frac{(m-1)r}{s}.
  \end{equation}


\end{theorem}
%

\subsection{Proof of Theorem \ref{th:general}}
It will be helpful to recall the notion of a {\em constructible subset} of an algebraic variety, see \cite[p. 393]{Lojasewick}: it is a set of the form
\[
\bigcup_{i=1}^r(V_i\setminus W_i),
\]
where $V_i,W_i$ are algebraic subvarieties of $X$. In other words, a constructible set is a finite union of quasialgebraic varieties. Chavalley's theorem \cite[p. 395]{Lojasewick} asserts that the projection of a constructible set is a constructible set.

We organize the proof of Theorem \ref{th:general} in several claims.
\begin{itemize}
 \item[Claim 1]: {\em There exists a proper subvariety
   $\mathcal{S}\subseteq\I$ such that for $p\in\I\setminus\mathcal S$
    {we have $n(p)=sd_\O$} and all the solutions $z\in\O$ of $\mathcal F(p,z)=0$ have multiplicity equal to
    $1$. {In particular, $\I\setminus\mathcal{S}$ has measure zero.} 
    Similarly, there exists at most a finite number of $z\in\O$ such that the degree of $\V_z$ is different from $r$}: this is a classical fact, but we include a short proof using tools from classical complex analytic geometry since the same tools will be used later.  
   The set
\[
 \hat \I=\{p\in\I:n(p)=sd_\O\}
\]
is constructible (see \cite[p. 398]{Lojasewick}). 
    {From \eqref{eq:condnew}, $\hat\I\neq\emptyset$, and furthermore,
    if $p\in\hat\I$, from B\'ezout's theorem \cite[p.
    432]{Lojasewick}, all the solutions $z\in\O$ of
    $\mathcal{F}(p,z)=0$ have multiplicity equal to $1$ (in
    particular, all $z$ are smooth points of $\O$). 
    Then, from the implicit function theorem for all $\tilde p$ sufficiently close to $p$ we
    have also $\tilde p\in\hat \I$. In other words, $\hat \I$ is a
    constructible open set, thus a quasialgebraic set. In particular
    its complement
    \[
 \mathcal{S}=\I\setminus\hat \I,
\]
is an algebraic subvariety of $\I$, and thus the first part of the claim
    follows.
    }
 The second claim is proved in a similar way: the property of $\V_z$
    having degree equal to $r$ is open and quasialgebraic, and thus it
    is satisfied in a Zariski open subset of $\O$. 
    {Since $\O$ is a one-dimensional irreducible algebraic variety,
    the last condition it is satisfied for all $\O$ except for, at most, a finite collection of points.}
\item [Claim 2]: {\em For all but a finite set of $z\in\O$, the set
\[
 \mathcal S_z=\{p\in\mathcal{S}:\mathcal{F}(p,z)=0\}
\]
is an algebraic variety of dimension at most $m-2$.} Assume that the set
    of $z$ such that $\mathcal S_z$ has dimension $m-1$ is infinite.
    Since the number of irreducible components of the $m-1$ dimensional
    variety $\mathcal{S}$ is finite (recall that any algebraic variety can be writen as a finite union of irreducible varieties, see for example \cite[p. 355]{Lojasewick}), this implies that there exists one
    such irreducible component $Y$ of dimension $m-1$ such that
    $\mathcal S_z\supseteq Y$ for an infinite collection of $z\in \O$.
    But then for all $p\in Y$ the number of solutions of
    $\mathcal{F}(p,z)=0$ is infinite, that is a contradiction to the hypotheses of the theorem.
\item [Claim 3]: {\em Let $\hat \O$  be the smooth part of $\O$ and let
  $\hat \V=\{(p,z)\in \hat \I\times \hat O:\mathcal F(p,z)=0\}$. Then,
    $\pi_1$ restricted to $\hat\V$ is a submersion onto $\hat\I$}. From
    Claim $1$, given $(p,z)\in\hat\V$, $z$ is a solution to $\mathcal
    F(p,z)=0$ of multiplicity $1$, that is {$\frac{\partial}{\partial
    z}\mathcal F(p,z)\neq 0$}  and from the implicit function theorem
    $\hat \V$ is a $m$--dimensional complex manifold, its tangent space is
\[
 T_{(p,z)}\hat\V={\mathrm Ker}D\mathcal{F}(p,z)\subseteq \I\times T_z\P(\C^b)\equiv\I\times z^\perp.
\]
and the projection $\pi_1\mid_{\hat \V}:\hat\V\to\hat\I$ is a submersion at $(p,z)$.
\item [Claim 4]: {\em The restriction of $\pi_2$ to $\hat \V$ is a submersion onto some Zariski open subset of $\hat\O$.} Note that $\pi_2:\V\to\O$ is surjective by the hypotheses of the theorem. 
  Now, from Claim 2 we have that for all but a finite number of points $z\in\hat\O$ 
  there is some
  $p\in\I$ such that {$p\in\V_z\setminus\mathcal S_z$}, and hence
  $\phi_2=\pi_2\mid_{\hat \V}$ contains $\hat \O$ except at most a finite number of points.
 Now, from Sard's theorem the set of singular values of $\phi_2$ is
 zero--measure, and it is also quasialgebraic since it is the intersection of $\hat\O$ with the projection of the algebraic variety $\{p,z:\mathcal F(p,z)=0, \frac{\partial}{\partial z}\mathcal F(p,z)=0\}$. It must then be at most a finite number of points, and the claim is proved.
\item [Claim 5]: {\em For any measurable mapping $g:\V\to[0,\infty]$:
\begin{equation}\label{eq:coarea}
 \int_{p\in \I}\sum_{z\in \O:\mathcal{F}(p,z)=0}g(p,z)\,dp=\int_{z\in
  \O}\int_{{p\in \V_z}}\frac{NJ\pi_1(p,z)}{NJ\pi_2(p,z)}g(p,z)\,dp\,dz,
\end{equation}
where $NJ$ means Normal Jacobian (i.e. the Jacobian of the derivative restricted to the orthogonal to its kernel). In particular, these integrals are well defined}. This is a consecuence of the smooth coarea formula applied to the two projections $\hat \V\to\hat \I$ and $\hat\V\to\hat\O$ (see for example \cite{BlCuShSm98} or \cite[p. 244]{Condition}). From Claim 3 we have:
\[
 \int_{(p,z)\in\hat \V}NJ\pi_1(p,z)\,g(p,z)\,d(p,z)=\int_{p\in \hat\I}\sum_{z\in \O:\mathcal{F}(p,z)=0}g(p,z)\,dp.
\]
On the other hand, from Claim 4 we have:
\[
 \int_{(p,z)\in\hat \V}NJ\pi_1(p,z)\,g(p,z)\,d(p,z)=\int_{z\in \hat\O}\int_{p\in \hat\I:\mathcal{F}(p,z)=0}\frac{NJ\pi_1(p,z)}{NJ\pi_2(p,z)}g(p,z)\,dp.
\]
Putting these two equalities together we have:
\[
 \int_{p\in \hat\I}\sum_{z\in \O:\mathcal{F}(p,z)=0}g(p,z)\,dp=\int_{z\in \hat\O}\int_{p\in \hat\I:\mathcal{F}(p,z)=0}\frac{NJ\pi_1(p,z)}{NJ\pi_2(p,z)}g(p,z)\,dp.
\]
Finally, from Claims 1 and 2 we can substitute $\hat\I$ and $\hat\O$ by $\I$ and $\O$ respectively in this last equality, and the claim follows.
\item [Claim 6]: {\em The quotient of the Normal Jacobians satisfies
\[
 \frac{NJ\pi_1(p,z)}{NJ\pi_2(p,z)}\mu(p,z)^2=\|p\|^2,\quad (p,z)\in\hat\V.
\]
}
{Given $(p,z)\in\hat\V$, let $D\pi_2(p,z)^*:T_z\hat\O\to T_{(p,z)}\hat\V$ }be the adjoint operator to $D\pi_2(p,z)$, and let
\[
 (\dot p_0,\dot z_0)=D\pi_2(p,z)^*(\dot z_0)\in T_{(p,z)}\hat\V,
\]
where $\dot z_0$ is any {non zero vector in $T_z\hat\O$}, and let
$(\dot p_j,\dot z_j),1\leq j\leq m-1$ be an orthonormal basis of the
orthogonal complement of $(\dot p_0,\dot z_0)$ in $T_{(p,z)}\hat\V$
(this tangent space was computed in Claim 3). Since the image of
$D\pi_2(p,z)^*$ is the orthogonal complement to the kernel of
$D\pi_2(p,z)$, then we have $\dot z_j=0$ and $\dot p_j$ orthogonal to $\dot p_0$ for $1\leq j\leq m-1$. 
Then, {writing the linear operators $D\pi_1(p,z)$ and $D\pi_2(p,z)$
on this basis we get}:
\[
NJ\pi_1(p,z)=|\det(D\pi_1(p,z))|^2=\frac{\|\dot p_0\|^2}{\|\dot
p_0\|^2+\|\dot z_0\|_z^2}.
\]
Similarly, we have 
\[
 NJ\pi_2(p,z)=\frac{\|\dot z_0\|_z^2}{\|\dot p_0\|^2+\|\dot z_0\|_z^2}.
\]
In particular,
\[
 \frac{NJ\pi_1(p,z)}{NJ\pi_2(p,z)}=\frac{\|\dot p_0\|^2}{\|\dot z_0\|_z^2}.
\]
Now, note from the definition of the condition number that $\mu(p,z)$ is
precisely $\|p\|$ times the operator norm of the linear operator
{
$D\mathrm{Sol}(p,z)$ given by
\[
   D\pi_2(p,z)\circ D\pi_1^{-1}(p,z):T_p\hat\I\to T_z\hat\O.
\]
}
Since we have {an orthogonal} basis $\dot p_0,\dot p_1,\ldots,\dot p_{m-1}$ such that
{$D\pi_2(p,z)\circ D\pi_1^{-1}(p,z)(\dot p_j)=0$} for $1\leq j\leq m-1$, we conclude that
{
\begin{align*}
  \frac{\mu(p,z)}{\|p\|}&=\|D\pi_2(p,z)\circ
 D\pi_1^{-1}(p,z)\|_2 \\
  & =\frac{\|D\pi_2(p,z)\circ D\pi_1^{-1}(p,z)(\dot p_0)\|_z}{\|\dot
  p_0\|}=\frac{\|D\pi_2(p,z)(\dot p_0,\dot z_0)\|_z}{\|\dot p_0\|}=\frac{\|\dot z_0\|_z}{\|\dot p_0\|}.
\end{align*}
}
We have thus proved the equality in the claim.
\item [Claim 7]: {\em The following equality holds:
\[
{\mathrm E} _{p\in\I}\left(\sum_{z\in\O:\mathcal{F}(p,z)=0}\mu(p,z)^2\right)=(m-1)\,r\,d_\O.
\]
(And this readily implies the theorem).
}
In order to prove Claim 7 we use claims 5 and 6 with $g(p,z)=e^{-\|p\|^2}\mu(p,z)^2/\pi^m$, getting:
\[
 {\mathrm E}
 _{p\in\I}\left(\sum_{z\in\O:\mathcal{F}(p,z)=0}\mu(p,z)^2\right)=\frac{1}{\pi^m}\int_{z\in
 \O}\int_{{p\in \V_z}}\|p\|^2e^{-\|p\|^2}\,dp\,dz.
\]
Now, for any $z\in\O$ (except at most a finite number) using Lemma \ref{lem:integral} we have
\begin{equation}\label{eq:aux}
  {
  \int_{p\in\V_z}\|p\|^2e^{-\|p\|^2}\,dp=
  \pi^{m-1}\,r\,(m-1),}
\end{equation}
since the set {$\V_z$} is (for almost all $z\in\O$) an algebraic subvariety of dimension $m-1$ and degree $r$. 
Using that the volume of $\O$ is $d_\O Vol(\P(\C^2))=\pi d_\O$ (see for example \cite[Cor. 20.10]{Condition}) we then get the claimed formula.
\end{itemize}

\section{Examples}
We now analyze some of the consequences of Theorem \ref{th:general}, visiting the scenarios {\bf A}, {\bf B}, {\bf C} and {\bf D} in the introduction.
\subsection{Solving univariate polynomials}\label{Dense}
Let $\mathcal{H}_N[X,Y]$ be the space of degree $N$ homogeneous polynomials with unknowns $X$ and $Y$. Let
\[
 \I=\mathcal{H}_N[X,Y],\quad \O=\P(\C^2),\quad \mathcal{F}(p,z)=p(z),
\]
so a pair $(p,z)$ is in $\V$ if and only if $z$ is a (projective) zero of $p$. Note that
the hypotheses of Theorem \ref{th:general} 
 are trivially satisfied. We then have:
\begin{equation}\label{eq:polinomios}
 {\rm E}_{p\in\mathcal{H}_N[X,Y]}\left(\frac{1}{N}\sum_{z\in\P(\C^2):\,p(z)=0}\mu_{*}(p,z)^2\right)=\frac{N}{N}=1,
\end{equation}
independently of the degree.
Since we are endowing the space $\mathcal{H}_N[X,Y]$ with the metric given by the
$2$--norm of the vector of coefficients of $p$ (which makes monomials of
different degrees {an orthonormal basis}), the condition number has the form
\[
 \mu(p,(z,1))=\frac{\|p\|\|(1,z,\ldots,z^N)\|}{|p'(z)|({1+|z|^2})}.
\]
A random polynomial is obtained by choosing coefficients $a_j$,
{associated to the monomial} $X^jY^{N-j}$ (with $0\leq j\leq N$),
{independent with distribution $\mathcal{N}_\C(0,1)$}. Equation \eqref{eq:polinomios} then implies that the condition number for such a random polynomial is quite small.
%

\subsection{Lacunary polynomial solving}\label{Lacunary}
Now, let $i=\{0,i_1,i_2,\ldots,i_k,N\}\subseteq\{0,\ldots,N\}$ (we identify the case $k=0$ with $i=\{0,N\}$) be a set of indices and let 
\[
 \mathcal{H}_N^{i}[X,Y]=\{a_0Y^N+a_{i_1}X^{i_1}Y^{N-i_1}+\cdots+a_{i_k}X^{i_k}Y^{N-i_k}+a_NX^N:a_0,a_{i_k},a_N\in\C\}
\]
be the space of polynomials containing only monomials of those degrees. Let 
\[
 \I=\mathcal{H}_N^{i}[X,Y],\quad \O=\P(\C^2),\quad \mathcal{F}(p,z)=p(z).
\]
The hypotheses of Theorem \ref{th:general} are again satisfied (note that the polynomial $X^N-Y^N$ has exactly $N$ projective zeros).

{As a subspace of $\mathcal{H}_N[X,Y]$, the input space
$\mathcal{H}_N^{i}[X,Y]$ can be endowed with the induced Hermitian
products given in the example before.}
Then, from Theorem \ref{th:general} we have
\[
 {\rm E}_{p\in \mathcal{H}_N^{i}[X,Y]}\left(\frac{1}{N}\sum_{z\in
 \P(\C^2):\, p(z)=0}\mu(p,z)^2\right)=\frac{k+1}{N}.
\]
This equality describes quantitatively how solving lacunary systems exhibit better stability properties than solving dense polynomials, which was qualitatively easily shown since the perturbations on the input are more restricted.

\subsection{Generalized eigenvalue problem}\label{GEVP}
As pointed out above, with
\[
 \I=\C^{n\times n}\times \C^{n\times n},\quad \O=\P(\C^2),\quad \mathcal{F}((A,B),(\alpha,\beta))=\det(\beta A-\alpha B),
\]
we get the GEVP. Let us now apply our Theorem \ref{th:general}. Note that the degrees of $\mathcal F$ in its two entries are $r=s=n$. We check that:
 \begin{itemize}
  \item The number of eigenvalues is finite for all nonsingular choices
    of $A,B$, that is out of a {$2n^2-2$} dimensional subvariety, and there exists $(A,B)$ such that the number of eigenvalues is equal to $n$.
\item For all $(\alpha,\beta)\in\P(\C^2)$, the set $\V_{(\alpha,\beta)}$
  of $(A,B)$ such that $\det(\beta A-\alpha B)=0$ is a {$2n^2-1$} variety. Moreover, $V_{(1,0)}$ is the set of $(A,B)$ such that $\det(B)=0$ which has degree $n$
 \end{itemize}
Thus, the hypotheses of Theorem \ref{th:general} are satisfied. We then have:
\[
 {\rm E}_{(A,B)\in \C^{n\times n}\times \C^{n\times
 n}}\left(\frac{1}{n}\sum_{(\alpha,\beta)\in {\mathrm{Eig}}(A,B)}\mu((A,B),(\alpha,\beta))^2\right)=\frac{(2n^2-1)n}{n}=2n^2-1.
\]
The expected value of the squared condition number for the GEVP is thus essentially equal to the size of the input.

\subsection{Polynomial eigenvalue problem}\label{PEVP}
This case is just:
\[
 \I=(\C^{n\times n})^{d+1},\quad \O=\P(\C^2),\quad \mathcal{F}(A,(\alpha,\beta))=\det\left(\sum_{j=0}^d\alpha^j\beta^{d-j}A_j\right),
\]
where $A=(A_0,\ldots,A_d)\in\I$. Again the hypotheses of Theorem \ref{th:general} are easily checked, and we conclude:
\[
  {\rm E}_{A\in (\C^{n\times n})^{d+1}}\left(\frac{1}{dn}\sum_{(\alpha,\beta)\in\mathrm{Eig}(A)}\mu(A,(\alpha,\beta))^2\right)=\frac{((d+1)n^2-1)n}{dn}=\frac{(d+1)n^2-1}{d}, \]
which proves our Theorem \ref{th:main}.

\subsection{Sparse polynomial eigenvalue problem}\label{sparsePEVP}
The same preceding argument can be applied to sparse versions of the PEVP case. Instead of writing down a generic result, we show how to use it in a particular case: assume for example that we deal with the Quadratic Eigenvalue Problem 
\[
\mathcal{F}((A,B,C),(\alpha,\beta))=\det(\alpha^2A+\alpha\beta B+\beta^2C)=0,
\]
where we impose on the input matrices some structure. For example, assume that $A$ is diagonal and $C$ is upper triangular, and assume that the nonzero entries follow again an independent distribution $\mathcal{N}_\C(0,1)$. We check the hypotheses of Theorem \ref{th:general}:
\begin{itemize}
	\item The number of solutions is finite as far as $A$ and $C$ are nonsingular, thus out of a codimension $2$ variety, and for input $A=\mathrm{Id}_n$ (the identity matrix), $B=0$ and generic $C$ the number of solutions is equal to $2n$. 
	\item $\V_{(\alpha,\beta)}$ is a codimension $1$ variety for all $(\alpha,\beta)\in\O$ and $\V_{(1,0)}$ is the set of $(A,B,C)$ such that $\det(A)=0$, that is a degree $n$ variety.
\end{itemize}
Theorem \ref{th:general} can thus be applied and we conclude that in this sparse case there are in general $2n$ eigenvalues and the expected condition number squared equals:
\[
{\rm E}_{A,B,C}\left(\frac{1}{2n}\sum_{(\alpha,\beta)\in\mathrm{Eig}(A,B,C)}\mu((A,B,C),(\alpha,\beta))^2\right)=\frac{(n+n^2+\frac{n(n+1)}{2})n}{2n}=\frac{3n^2+3n}{4}. \]
\subsection{More elaborated eigenvalue problems}\label{ellaboratedPEVP}
We finally include one example of use in the case that $\O\neq\P(\C^2)$. Assume that our problem is the following: on input $A,B,C$ (three $n\times n$ matrices), find $(\alpha,\beta,\gamma)\in\P(\C^3)$ such that:
\[
\alpha\beta+\alpha\gamma+\beta\gamma=0,\quad \det(\alpha A+\beta B+\gamma C)=0.
\]
This is, in some sense, a system of two homogeneous equations and three variables, thus one would expect the solution set to be a finite collection of points $(\alpha,\beta,\gamma)\in\P(\C^3)$. We can treat this problem inside our framework as follows. Let
\[
\O=\{ (\alpha,\beta,\gamma)\in\P(\C^3):\alpha\beta+\alpha\gamma+\beta\gamma=0\},
\]
which is a degree $2$ irreducible projective algebraic subvariety of $\P(\C^3)$. Assume now that $A,B,C$ have random independent entries with distribution $\mathcal{N}_\C(0,1)$. We check that the hypotheses of Theorem \ref{th:general} hold:
\begin{itemize}
	\item If $A$ is nonsingular, then $(1,0,0)\in\O$ is not a solution of 
	\[
	\mathcal F((A,B,C),(\alpha,\beta,\gamma))=\det(\alpha A+\beta B+\gamma C)=0.
	\] 
	Similarly, if $B$ is nonsingular then $(0,0,1)$ is not a solution. Since for any input the solution set is an algebraic subvariety of $\O$ (i.e. the total set or a finite collection of points), we conclude that out of a codimension $2$ subvariety the number of solutions is finite. Moreover, for $(A,B,C)=(\mathrm{Id}_n,B,-B)$ where $B$ is diagonal with entries $1,2,\ldots,n$ the solutions are those points of the form
	\[
	(-m,\gamma+1,\gamma)\in\O,\quad 1\leq m\leq n,\quad \beta\in\C.
	\]
	For these points to be in $\O$ we need
	\[
	-m(\gamma+1+\gamma)+(\gamma+1)\gamma=0,
	\]
	that is for each $m$ there are exactly $2$ values of $\gamma$ that make $	(-m,\gamma+1,\gamma)\in\O$. We thus have found an input with exactly $2n$ solutions.
	\item $\V_{(\alpha,\beta,\gamma)}$ is a codimension $1$ variety for all $(\alpha,\beta,\gamma)\in\O$ and $\V_{(1,0,0)}$ is the set of $(A,B,C)$ such that $\det(A)=0$, that is a degree $n$ variety.
\end{itemize}
We can thus apply Theorem \ref{th:general} and we conclude that the expected condition number squared for the problem in this section equals:
\[
{\rm E}_{A,B,C}\left(\frac{1}{2n}\sum_{(\alpha,\beta,\gamma)\in\O:\mathcal F((A,B,C),(\alpha,\beta,\gamma))=0}\mu((A,B,C),(\alpha,\beta,\gamma))^2\right)=\frac{(3n^2-1)n}{n}=3n^2.
\]

\appendix
\section{A useful integral}
In this section we prove the following lemma.
\begin{lemma}\label{lem:integral}
 Let $\J\subseteq\C^a$ be a homogeneous complex algebraic variety of degree $d$ and dimension $n$. Then,
 \[
  \int_{p\in\J}\|p\|_2^2e^{-\|p\|_2^2}\,dp=\pi^{n}\,n d.
 \]
\end{lemma}
\begin{proof}
 We consider the projection $\pi:\J\to\P(\J)$ whose Normal Jacobian is equal to 
 \[
  NJ(\pi)(p)=\frac{1}{\|p\|_2^{2n-2}}.
 \]
(In order to compute this Normal Jacobian just consider any o.n. basis of $T_p\J$ whose first vector is $\dot p=p$).

The Coarea Formula then yields
\[
 \int_{p\in\J}\|p\|_2^2e^{-\|p\|_2^2}\,dp=\int_{q\in\P(J)}\int_{p\in\J:\pi(p)=q}\|p\|_2^{2n-2+2}e^{-\|p\|_2^2}\,dp\,dq=
 \]
 \[
 Vol_{\P(\C^a)}(\P(\J))\int_{w\in\C}|w|^{2n}e^{-|w|^2}\,dw
\]
Since $\P(\J)$ is a degree $d$ and dimension $n-1$, from \cite[Cor. 20.10]{Condition} its volume in the projective space equals
\[
 dVol(\P(\C^n))=\frac{\pi^{n-1}}{\Gamma(n)}\,d.
\]
On the other hand, using polar coordinates we have
\[
 \int_{w\in\C}|w|^{2n}e^{-|w|^2}\,dw=2\pi\int_0^\infty t^{2n+1}e^{-t^2}\,dt=\pi\int_0^\infty s^{n}e^{-s}\,ds=\pi\Gamma(n+1).
\]
The lemma follows.

\end{proof}


\begin{thebibliography}{10}
\bibitem{Armentano:10}
D.~Armentano.
\newblock Stochastic perturbations and smooth condition numbers.
\newblock {\em J. Complexity}, 26 (2), 2010, pp. 167--171.

  \bibitem{ABBCS} 
    D.~Armentano, C.~Beltr\'an, P.~B\"urgisser, F.~Cucker, M.~Shub. 
    \emph{A stable, polynomial-time algorithm for the eigenpair problem.} 
    To appear in Journal of the European Mathematical Society, available at {\tt arXiv:1410.0116}. 
  \bibitem{ABBCS2} 
    D.~Armentano, C.~Beltr\'an, P.~B\"urgisser, F.~Cucker, M.~Shub. 
    \emph{Condition lenght and complexity for the solution of polynomial systems} 
   {\em Found. Comput. Math.}, 16(6), 2016, pp. 1401-1422.  

    
    \bibitem{ArmCuc}
D.~Armentano and F.~Cucker. 
\newblock A randomized homotopy for the Hermitian eigenpair problem.
\newblock {\em Found. Comput. Math.}, 15 (1), 2015, pp. 281--312.  


    \bibitem{Aurentz}
    J. L. Aurentz, R. Vandebril and D.S. Watkins. Fast computation of the zeros of a polynomial via factorization of
the companion matrix, {\em  SIAM J. Sci. Comput.}, { 35}, 2013, pp. A255--A269.

     \bibitem{BaiSu} Z.Bai and Y. Su. SOAR: a second-order Arnoldi method for the solution of the quadratic eigenvalue problem. {\em SIAM J. Matrix Anal. Appl.} { 26} (3), 2005, pp. 640--659.
    
    
    \bibitem{Barel}
    M. Van Barel, M. Vandebril, R. Van Dooren and K. Frederix. Implicit double shift QR-algorithm for companion
matrices, {\em Numer. Math.}, { 116}, 2010, pp. 177--212.
    
    \bibitem{Facility}
    C. Beltr\'an. A facility location formulation for stable polynomials and elliptic Fekete points. {\em Found. Comput. Math. } 15(1). 2015, pp. 125--157.
    
    
\bibitem{BePa:11}
C.~Beltr{\'a}n and L.M. Pardo. 
\newblock Fast linear homotopy to find approximate zeros 
of polynomial systems.
\newblock {\em Found. Comput. Math.}, 11(1):95--129, 2011.


\bibitem{Bini1}
D.A. Bini, L. Gemignani and V.Y. Pan. { Fast and stable QR eigenvalue algorithms for generalized companion matrices
and secular equations.} {\em Numer. Math.}, { 100}, 2005, pp. 373--408.

\bibitem{Bini2}
D.A. Bini, P. Boito, Y. Eidelman, L. Gemignani and I. Gohberg. A fast implicit QR eigenvalue algorithm for
companion matrices, {\em Linear Algebra Appl.}, { 432}, 2010, pp. 2006--2031.

  \bibitem{BlCuShSm98}
    L.~Blum, F.~Cucker, M.~Shub, and S.~Smale.  
    \emph{Complexity and real  computation}, Springer-Verlag, New York, 1998.
    
  \bibitem{Condition}
    P.~B{\"u}rgisser and F.~Cucker. 
    \emph{Condition}, volume 349 of 
    \emph{Grundlehren der mathematischen Wissenschaften}. 
    Springer-Verlag, Berlin, 2013.

    \bibitem{ChenDongarra2005}
{Z.  Chen and J. J. Dongarra}. 
\newblock Condition numbers of {G}aussian random matrices.
\newblock {\em SIAM J. Matrix Anal. Appl.}, { 27} (3), 2005, pp. 603--620 (electronic).
    
\bibitem{Demmel}
J.~W. Demmel. 
\newblock On condition numbers and the distance to the nearest ill-posed problem.
\newblock {\em Numerische Mathematik\/}, 51 (3), 1987, pp. 251--289

\bibitem{DeTi}
J.-P. Dedieu and F. Tisseur. 
\newblock Perturbation theory for homogeneous polynomial eigenvalue problems.
\newblock {\em Linear Algebra Appl.} 358, 2003, pp. 71--94.

\bibitem{DeSh00}
J.-P. Dedieu and M.~Shub.
\newblock Multihomogeneous {N}ewton methods.
\newblock {\em Math. Comp.}, 69 (231), 2000, pp. 1071--1098.

\bibitem{Dopico}
F. De Ter\'an, F.M. Dopico and J. Pérez. Backward stability of polynomial root-finding using Fiedler companion matrices, {\em IMA J. Numer. Anal.} {  36} (1). 2016, pp. 133--173.

\bibitem{Edelman}
{ A. Edelman.}
\newblock Eigenvalues and condition numbers of random matrices.
\newblock {\em SIAM J. Matrix Anal. Appl.}, { 9} (4) 1988, pp. 543--560.

\bibitem{Edelman-Murakami}
A. Edelman and H. Murakami. Polynomial roots from companion matrix eigenvalues. {\em Math. Comp.} {  64} (210), 1995, pp. 763--776.

\bibitem{Higham3}
L. Grammont, N.J. Higham and F. Tisseur. 
A framework for analyzing nonlinear eigenproblems and parametrized linear systems. 
 {\em Linear Algebra Appl.} { 435} (3), 2011, pp. 623--640. 

\bibitem{Higham1}
N.J. Higham, D.S. Mackey, F and F. Tisseur. The conditioning of linearizations of matrix polynomials. {\em
SIAM J. Matrix Anal. Appl.} { 28} (4), 2006, pp. 1005--1028.


\bibitem{Higham2}
N.J. Higham, R--C. Li and F. Tisseur. Backward error of polynomial eigenproblems solved by linearization. {\em SIAM J. Matrix Anal. Appl.} {  29} (4), 2007, pp. 1218--1241. 

\bibitem{Tisseur2013}{
   {S. Hammarling, C.J. Munro and F. Tisseur}. 
   {An algorithm for the complete solution of quadratic eigenvalue
   problems},
   {\em ACM Trans. Math. Software},
   { 39}(3), {2013}.
}


\bibitem{Zaballa}
P. Lancaster and I. Zaballa. 
On the inverse symmetric quadratic eigenvalue problem. {\em SIAM J. Matrix Anal. Appl.} {  35} (1), 2014, pp. 254--278

  \bibitem{Lojasewick}
    S.~Lojasiewick.  
    \emph{Introduction to Complex Analytic Geometry}, Basel--Boston--Berlin, 1991.

    
    \bibitem{Meerbergen}
    K. Meerbergen. The quadratic Arnoldi method for the solution of the quadratic eigenvalue problem. {\em SIAM J. Matrix Anal. Appl.} { 30} (4), 2008/09, pp. 1463--1482
    
    
    
    
    \bibitem{MehrmannVoss}{
   V. Mehrmann and H. Voss. 
   {Nonlinear eigenvalue problems: a challenge for modern eigenvalue
   methods},
   {\em GAMM Mitt. Ges. Angew. Math. Mech.},
   { 27} (2),
   {2005},
}

\bibitem{Moler}
C. Moler. Cleve’s corner: Roots–of polynomials, that is. {\em The Mathworks Newsletter}, { 5}, 1991, pp. 8--9.
    
    \bibitem{Tisseur2000}
    F. Tisseur. Backward error analysis of polynomial eigenvalue problems. {\em Linear Algebra Appl.} 309, 2000, pp.  339--361, 2000.
    
\bibitem{Trefethen}
L. Trefethen and D. Bau III. \emph{Numerical Linear Algebra.} 
SIAM, Philadelphia, 1997.    
    
    
 \bibitem{ShSm93b}
    M.~Shub and S.~Smale.  
    \emph{Complexity of B\'ezout's theorem. I. Geometric aspects.}
    J. Amer. Math. Soc. 6 (2), 1993, pp. 459--501.
    
    \end{thebibliography}
\end{document}